\documentclass[12pt]{amsart}
\usepackage{amssymb,amsfonts,amsmath,amsopn,amstext,amscd,color,latexsym,mathabx,mathrsfs,skak,stackrel,xy,url,verbatim}

\input xy
\xyoption{all}

\theoremstyle{remark}

\newtheorem{para}{\bf}[section]

\theoremstyle{definition}

\theoremstyle{plain}

\newtheorem{thm}[para]{Theorem}

\newtheorem{lemma}[para]{Lemma}
\newtheorem{cor}[para]{Corollary}
\newtheorem{prop}[para]{Proposition}

\newenvironment{numequation}{\addtocounter{para}{1}
\begin{equation}}{\end{equation}}

\newcommand{\bbC}{{\mathbb C}}

\newcommand{\bbF}{{\mathbb F}}

\newcommand{\bbK}{{\mathbb K}}

\newcommand{\bbQ}{{\mathbb Q}}

\newcommand{\bbZ}{{\mathbb Z}}

\newcommand{\bA}{{\bf A}}

\newcommand{\bG}{{\bf G}}

\newcommand{\bP}{{\bf P}}

\newcommand{\cB}{{\mathcal B}}

\newcommand{\cF}{{\mathcal F}}
\newcommand{\cG}{{\mathcal G}}

\newcommand{\cL}{{\mathcal L}}

\newcommand{\cO}{{\mathcal O}}

\newcommand{\cS}{{\mathcal S}}

\newcommand{\sH}{{\mathscr H}}

\newcommand{\sL}{{\mathscr L}}

\newcommand{\End}{{\rm End}}

\newcommand{\Pf}{{\it Proof. }}

\newcommand{\Spec}{{\rm Spec}}

\newcommand{\car}{\stackrel{\simeq}{\longrightarrow}}

\newcommand{\LG}{LG}

\newcommand{\hT}{\tilde{T}}

\newcommand{\KTs}{\bbK'_{\tilde{T}}}

\newcommand{\ke}{k}
\newcommand{\kee}{\overline{\bbF}_p}

\begin{document}

\title{Hecke algebras and affine flag varieties in characteristic $p$}
\author{Tobias Schmidt}
\address{Institute de Recherche Math\'ematique de Rennes, 
Universit\'e de Rennes 1, Campus de Beaulieu, 35042 Rennes, France}
\email{Tobias.Schmidt@univ-rennes1.fr}

\thanks{The author would like to acknowledge support of the Heisenberg programme of Deutsche Forschungsgemeinschaft (SCHM 3062/1-1).}

\begin{abstract} Let $G$ be a split semi-simple $p$-adic group and let $\sH$ be its Iwahori-Hecke algebra with coefficients in the algebraic closure $\kee$ of $\bbF_p$. Let $\cF$ be the affine flag variety over $\kee$ associated with $G$.
We show, in the simply connected simple case, that a torus-equivariant $K'$-theory of $\cF$ (with coefficients in $\kee$) admits an action of $\sH$ by Demazure operators and that this provides a model for the regular representation of $\sH$.

\end{abstract}
\maketitle

\tableofcontents

\section{Introduction}

Let $F$ be a $p$-adic local field with ring of integers $o_F$ and let $\kee$ be the algebraic closure of its residue field. Let $G$ be a connected reductive linear algebraic group over $F$ and denote by $G(F)$ the group of its $F$-rational points. For a choice of Iwahori subgroup $I\subset G(F)$, the corresponding Iwahori-Hecke algebra $\sH_{\kee}$ over $\kee$ equals the convolution algebra $\kee[I\setminus G(F)/I]$ of $\kee$-valued functions with finite support on the double cosets of $G(F)$ mod $I$. The algebra $\sH_{\kee}$ and a certain extension $\sH_{\kee}^{(1)}$ of it associated to the maximal pro-$p$ subgroup of $I$ - the pro-$p$ Iwahori-Hecke algebra -  were systematically studied by Vign\'eras \cite{VignerasProp1}, \cite{VignerasGeneric}, \cite{VignerasProp2}. Among their first applications is the mod $p$ local Langlands program for the group $G(F)$ \cite{GK_phigamma1}.
\vskip8pt

In this note we only deal with $\sH_{\kee}$ and realize it as an equivariant algebraic $K$-theory of its affine flag variety $\cF$ over $\kee$. 
We view this result as a first step to analyze the mod $p$ geometry of the algebras $\sH_{\kee}$ and $\sH_{\kee}^{(1)}$. In particular, we hope that the geometry of affine Springer fibres in $\cF$ \cite{Goertz_Affine} will eventually shed a light on the representation theory of these algebras, in analogy to the classical case of the finite complex flag variety \cite{KLDeligneLanglands}. We also note that the occurence of affine geometry in the mod $p$ setting can be viewed as a consequence of the failure of the Bernstein presentation for the algebras $\sH_{\kee}$ and $\sH_{\kee}^{(1)}$ \cite{VignerasGeneric}.

\vskip8pt

 To give more details about the present article, let us assume that $G$ extends to a split simply connected semisimple and simple group scheme over the integers. We denote its base change to $\overline{\bbF}_p$ by the same letter. Let $T\subset G$ be a maximal torus. We choose an Iwahori subgroup $\cB\subset G(\overline{\bbF}_p[[t]])$ whose projection to $G(\overline{\bbF}_p)$ contains $T$ and let $$\cG=G(\overline{\bbF}_p((t)))/G(\overline{\bbF}_p[[t]]) \hskip10pt {\rm resp.}\hskip10pt \cF=G(\overline{\bbF}_p((t)))/\cB$$ be the affine Grassmannian resp. the affine flag variety of $G$ over $\kee$. They are ind-schemes whose structure can be defined through the increasing union of Schubert varieties $S_w$ indexed by elements $w$ in the algebraic cocharacters $\Lambda$ of $T$ resp. in the affine Weyl group $W$ of $G$. By definition, $S_w$ equals the closure of the $\cB$-orbit through $w*$ and has a natural structure as finite dimensional projective variety over $\kee$. Let $\hT$ be the extension of $T$ by the 'turning torus' $\bG_m$ (which corresponds to the derivation element $d$ in the Kac-Moody setting). Then $\hT$ acts on $\cG$ and $\cF$ and preserves the stratification into Schubert varieties. Denote by $K'_{\hT}(S_w)$ the Grothendieck group of $\hT$-equivariant coherent module sheaves on $S_w$ and put $\KTs(\cdot):=K'_{\hT}(\cdot)\otimes\kee.$ We define $$\KTs(\cF):=\varinjlim_w \KTs(S_w)$$ and show that the Schubert classes $[\cO_{S_w}]$ provide an $\kee[\hT]$-basis for $\KTs(\cF)$. Here, $\kee[\hT]$ denotes the ring of algebraic functions on $\hT$ over $\kee$. The analogous result holds for $\cG$. Letting $\pi:\cF\rightarrow\cG$ denote the projection, the pull-back
$$ \pi^*: \KTs(\cG)\hookrightarrow  \KTs(\cF)$$
is a linear injection whose image can be described explicitly.
To each simple affine reflection $s$ corresponds a smooth Bruhat-Tits group scheme $P_s$ over $\overline{\bbF}_p[[t]]$ with connected fibres and with corresponding partial affine flag variety $\cF_s$. The projection $\pi_s:\cF\rightarrow\cF_s$ is a $\bP^1$-bundle. One obtains a Demazure operator $D_s:=\pi_s^*\pi_{s*}$ on $\KTs(\cF)$ in analogy to the classical case of the finite flag variety treated in \cite{DemazureSchubert}. The operators $D_s$ satisfy braid relations and quadratic relations which are exactly the ones appearing in the Iwahori-Matsumoto presentation of $\sH_{\kee}$. The latter is described in terms of the $\kee$-basis given by the characteristic functions $\tau_w$ of the double cosets $IwI$ in the classical way \cite{IwahoriMatsumoto}, however in our situation, the parameters satisfy $q_s=0$ for all $s$, since $\kee$ has characteristic $p$. Letting $\tau_s$ act through $D_s$ therefore turns $\KTs(\cF)$ into a right module over $\sH_{\kee[\hT]}=\sH_{\kee}\otimes \kee[\hT]$. In this situation, our main result says that mapping the function $(-1)^{\ell(w)}\tau_w$ to the Schubert class $[\cO_{S_w}]$ for each $w$ induces an isomorphism
 $$ \sH_{\kee[\hT]}\car \KTs(\cF)$$
  as right $\sH_{\kee[\hT]}$-modules. Here, $\ell$ denotes the length function on $W$. Finally, specialising the isomorphism at the unit element $\hT=1$ yields an isomorphism between $\sH_{\kee}$ and the absolute $K'$-theory $\bbK'(\cF)$ (with coefficients in $\kee$) as right $\sH_{\kee}$-modules. In particular, any irreducible right $\sH_{\kee}$-module appears as a quotient of $\bbK'(\cF)$. As a corollary, we show that the restriction of $\pi^*$ to the module spanned by Schubert classes supported on anti-dominant cocharacters $\Lambda_-$ realizes the regular representation for the spherical Hecke algebra $\kee[G(o_F)\setminus G(F)/G(o_F)]$ viewed inside $\sH_{\kee}$ via the mod $p$ Satake isomorphism \cite{HerzigSatake}. We remark that all our results hold in fact with $\kee$ replaced by any subfield $k\subseteq\kee$.

\vskip8pt

We expect all results to extend to arbitrary connected and split reductive groups. Moreover, replacing the Iwahori-subgroup $\cB\subset LG$ by the inverse image of the unipotent radical of the Borel subgroup of $G$ in question should extend all results from $\sH_{\kee}$ to $\sH_{\kee}^{(1)}$.
We leave these issues for future work.

\vskip8pt
To make a brief comparison with the complex case, we recall that over the complex numbers
the link between representations of the Iwahori-Hecke algebra and equivariant $K$-theory was discovered by Lusztig \cite{LusztigAMS} and occupies a central position in applications to the local Langlands program. We only mention the proof of the Deligne-Langlands conjecture by Kazhdan-Lusztig \cite{KLDeligneLanglands} where it is shown that the equivariant $K$-homology of the complex Steinberg variety of triples affords a natural Hecke action through $q$-analogues of Demazure operators and that this provides a model for the regular representation of the Hecke algebra.
We also point out that the theory of Demazure operators on equivariant $K$-theory of affine flag manifolds is well-established in the complex Kac-Moody setting, due to the work of Kostant-Kumar \cite{KoKu2} and Kashiwara-Shimozono \cite{KashiwaraFlag}.

\vskip8pt

{\it Acknowledgements:} The author would like to thank Elmar Gro\ss e-Kl\" onne and Marc Levine for some interesting conversations related to this topic.

\section{Algebraic loop groups and their affine flag varieties}

 We recall some definitions and concepts from \cite{Faltings}, \cite{PR_twisted}. Let $k$ be a field and let $K=k((t))$ be the field of formal Laurent series over $k$. Let $G$ be a simply connected semi-simple and simple group over $k$ which is split over $k$. Let $LG$ be its associated algebraic {\it loop group}, i.e. the functor $$R\mapsto LG(R)=G(R((t)))$$ on the category of $k$-algebras. It is representable by an ind-scheme over $k$.
 Let $L^+G$ be its associated {\it positive loop group} which is an affine group scheme over $k$ representing the functor $$R\mapsto L^+G(R)=G(R[[t]]).$$
 Let $\cB$ be the inverse image under the morphism ${\rm pr}: L^+G\rightarrow G, t\mapsto 0$ of a Borel subgroup in $G$. Thus, $\cB$ equals a closed subscheme of an infinite dimensional affine $k$-space. The {\it affine Grassmannian} $\cG$ resp. the {\it affine flag  variety} $\cF$ of $LG$ is the ind-scheme over $k$ $$\cG=LG/L^+G\hskip15pt {\rm resp.}\hskip15pt \cF=LG/\cB$$ representing the functor $R\mapsto LG(R)/L^+G(R)$ resp. $R\mapsto LG(R)/\cB(R).$ The natural projection map $\pi: \cF\rightarrow\cG$ is a proper and Zariski locally trivial fibration with fibers $L^+G/\cB\simeq G/{\rm pr}(\cB)$ (the split case in \cite[8.e.1]{PR_twisted}). The group $LG$ acts on $\cG$ and $\cF$ by left translations. Evaluating $G$ on the structure map $R\mapsto R((t))$ gives an inclusion $G\subset LG$. In this way, $G$ and all its subgroups act on $\cG$ and $\cF$.

 \vskip8pt

 For example, let $G={\rm SL}_d$, the special linear group. Then $LG$ equals the union of subfunctors $L_nG$ where $L_nG(R)$ is the set of matrices $M=\sum_{i\geq -n}M_it^{i}$
 of determinant $1$ with $M_i$ a $d\times d$-matrix with entries in $R$. Each $L_nG$ is representable as a closed subscheme of the affine scheme $\prod_{i\geq -n}\bA^{d^2}_k$
 and the inclusion $L_nG\subset L_{n+1}G$ is a closed embedding. This defines the ind-structure on $LG$. To make the ind-structure of $\cG$ and $\cF$ explicit, recall that
 a {\it lattice}  $\sL\subset R((t))^d$ is a $R[[t]]$-submodule of $R((t))^d$ such that there is $n\geq 0$ with $$t^{n}R[[t]]^d\subseteq\sL\subseteq t^{-n}R[[t]]^d$$ and such that the quotient $t^{-n}R[[t]]^d/\sL$ is a locally free $R$-module of finite rank. If $\bigwedge^d\sL=R[[t]]$, the lattice is called {\it special}. We denote the set of special lattices over $R$ by $\cL(R)$ and those for fixed $n$ by $\cL_n(R)$. Then $\cL_n$ is representable by a projective $k$-scheme. Indeed, the morphism $$\cL_n(R)\rightarrow {\rm Grass}_{nd}(t^{-n}k[[t]]^d/t^nk[[t]]^d)(R),\hskip5pt \sL\mapsto \sL/t^nR[[t]]^d$$ defines a closed embedding into the Grassmannian of $nd$-dimensional subspaces of the $2nd$-dimensional $k$-vector space
 $t^{-n}k[[t]]^d/t^nk[[t]]^d.$ Note that multiplication by $1+t$ induces a (unipotent) automorphism on the latter vector space which induces an automorphism of the Grassmannian. The image of $\cL_n$ equals its fixed point variety and hence, is indeed closed. Each element of $LG(R)$ gives rise to a special lattice in $R((t))^d$ by applying it to the standard lattice $R[[t]]^d$. Since $G(R[[t]])$ equals the stabilizer of the standard lattice, we obtain an isomorphism
 $ \cG(R)\simeq\cL(R)$. This defines the ind-structure of $\cG$ in the case of ${\rm SL}_d$. Now a {\it lattice chain} in $R((t))^d$ is a chain
 $$ \sL_0\supset \sL_1\supset\cdot\cdot\cdot\supset\sL_{d-1}\supset t\sL_0$$
 with lattices $\sL_i$ in $R((t))^d$ such that each quotient $\sL_i/\sL_{i+1}$ is a locally free $R$-module of rank $1$. Each element of $LG(R)$ gives rise to a lattice chain with $\sL_0$ special by applying it to the standard lattice chain
 $$ \sL_i:=\bigoplus_{j=1,...,d-i} R[[t]]e_{j} \oplus \bigoplus_{j=d-i+1,...,d}tR[[t]]e_{j}$$
 where $e_1,...,e_d$ denotes the standard basis of $R((t))^d$. Since $\cB(R)$ equals the stabilizer of the standard lattice chain, we obtain an isomorphism between
 $ \cF(R)$ and the space of lattice chains $\sL_\bullet$ in $R((t))^d$ such that $\sL_0$ is special. In this optic, the projection $\pi$ is given by sending a chain $\sL_\bullet$ to $\sL_0$ and, hence, is a fiber bundle whose fibres are all isomorphic to the usual flag variety $G/{\rm pr}(\cB)$ of $G$. This defines the ind-structure of $\cF$ in the case of ${\rm SL}_d$. For more details we refer to \cite{Goertz_Affine}.

 \vskip8pt

In the following, we are forced to work with a less explicit realization of the ind-structure on $\cF$. To define it requires more notation. Our assumptions on $G$ imply that $\cB(k)$ is an Iwahori-subgroup of $G(K)$ and hence, defines a fundamental chamber in the Bruhat-Tits building of $G(K)$. We denote by $B$ the associated smooth group scheme with connected fibres over $k[[t]]$. 
Then $\cB=L^+B$ in the notation of \cite[1.b]{PR_twisted}. Similarly, to any wall $F_i$ (i.e. facet of codimension $1$) of this chamber, we let $P_{i}$ be the smooth group scheme with connected fibres over $k[[t]]$ such that $P_i(k[[t]])$ is equal to the parahoric subgroup of $G(K)$ associated to $F_i$. To $P_{i}$ corresponds an infinite-dimensional affine group scheme $L^+P_{i}$ over $k$ with $L^+P_{i}(R)=P_{i}(R[[t]])$. The fpqc-quotient $\cF_i:=LG/L^+P_i$ is representable by an ind-scheme over $k$, the {\it partial affine flag variety} associated to $F_i$. The projection morphism \begin{numequation}\label{projection}\pi_i:\cF\rightarrow\cF_i
\end{numequation} is a principal $L^+P_i/L^+B$-bundle for the Zariski topology (the split case in \cite[Thm. 1.4]{PR_twisted}). Moreover, $L^+P_i/L^+B\simeq \bP^1_k$, the projective line over $k$, according to \cite[Prop. 8.7]{PR_twisted}. We denote the reflection in the wall $F_i$ by $s_i$.

\vskip8pt

Let $T$ be a maximal torus in $G$ contained in the Borel subgroup ${\rm pr}(\cB)$ and let $N_G(T)$ be its normalizer. Let $W_0=N_G(T)/T$ be the finite Weyl group of $G$ with longest element $w_0$. It may be viewed as the group generated by reflections in $\cS_0:=\{s_1,...,s_l\}.$
Let $\Phi=\Phi(G,T)$ be the root system and let $\Phi^+$ be the set of positive roots distinguished by ${\rm pr}(\cB)$. Let $\Lambda:=X_*(T)$ denote the group of cocharacters of $T$. We view $\Lambda$ as a subgroup of $T(k[[t]])$ via the mapping $\lambda\mapsto \lambda(t)$. In particular, $\Lambda$ acts on $\cG$ and $\cF$.

Let $W$ be the affine Weyl group of $G$. The group $W$ is an affine Coxeter group with set of generators $\cS:=\cS_0 \cup \{s_0\}$. We denote by $\leq$ the Bruhat-Chevalley (partial) order on $W$ and by $\ell(\cdot)$ the length function.
Since $G$ is simply connected, $\Lambda$ coincides with the coroot lattice whence $W$ may be written as the semi-direct product $W= \Lambda \rtimes W_0.$ When considering $\Lambda$ as a subgroup of $W$ (whose group law is written multiplicatively), we will use exponential notation.

\vskip8pt
Let $P$ be either $B$ or $P_i$ or $G$ and write for a moment $\cF_P=LG/L^+P$. Denote by $*=L^+P\in\cF_P$ the base point. Let $W^P$ be the set of representatives in $W$ of minimal length for the right cosets $W/W_P$ where $W_P\subset W$ equals the subgroup generated by $1$ or $s_i$ or $\cS_0$.
The {\it $(B,P)$-Schubert cell} $C_w=C_w(B,P)$ is the reduced subscheme

$$ L^+B w* \subset \cF_P.$$

The {\it $(B,P)$-Schubert variety} $S_w=S_w(B,P)$ is the reduced subscheme with underlying set
the Zariski closure of $C_w$ in $LG/L^+P$. It is a projective variety over $k$. Given $u\leq w$ there is a closed immersion $S_u\hookrightarrow S_w$ and
$$LG/L^+P=\bigcup_{w\in W^P} S_w$$ defines the ind-structure on $\cF_P$.
If we let $\cF_P^n$ be the (finite) union over the $S_w$ with $\ell(w)\leq n$, then $\cF_P=\cup_n \cF_P^n$ defines the ind-structure on $\cF_P$, too.
In case of $P=P_i$, we shall write $\cF_i=\cup_n \cF_i^n$. In the following, we sometimes drop the base point $*$ from the notation.

\vskip8pt

Let $\Delta=\{\alpha_1,..,\alpha_l\}$ be the set of simple roots in $\Phi$ associated with $\cS_0$ and let $\theta$ be the highest root. Denote by $\Phi^{\rm aff}$ the collection of affine roots $(\alpha,m)$ with $\alpha\in\Phi$ and $m\in\bbZ$. Then $\alpha_0:=(-\theta,1)$ is
the remaining affine simple root which induces the reflection $s_0$. For each affine root $(\alpha,m)$, there is an associated inclusion
$$\phi_{(\alpha,m)}: SL_2\hookrightarrow LG$$ defined as follows \cite[(9.8)]{PR_twisted}. The vector part $\alpha$ defines a $k$-morphism $SL_2\hookrightarrow G$ in the usual way which induces by functoriality $L(SL_2)\hookrightarrow LG$. The map $\phi_{(\alpha,m)}$ equals the composition with the morphism
$$ SL_2\longrightarrow L(SL_2), \left(
                                  \begin{array}{cc}
                                    a & b \\

                                    c & d \\
                                  \end{array}
                                \right)\mapsto  \left(
                                  \begin{array}{cc}
                                    a & bt^m \\

                                    ct^{-m} & d \\
                                  \end{array}
                                \right).$$

The image of $\phi_{\alpha_i}$ is contained in a Levi subgroup of the parabolic $L^+P_i\subset LG$ for $i=0,...,l$. There is an irreducible $k$-linear representation $V_i$ of the group scheme $L^+P_i$ whose pull-back via $\phi_{\alpha_i}$ equals the unique irreducible representation of $SL_2$ of dimension $2$. Let $LG\times^{L^+P_i}\bP(V_i)$ be the associated projective bundle over $\cF_i$. Let $v_i$ be a highest weight vector in $V_i$ and for $g\in LG$ denote by $[g,v_i]$ the class of $(g,kv_i)\in LG\times^{L^+P_i} \bP(V_i)$. We then have a morphism of $\bP^1$-bundles $\cF\rightarrow LG\times^{L^+P_i}\bP(V_i)$ given by $gL^+B\mapsto [g,v_i]$.
\begin{lemma}\label{lem-ident} The morphism $\cF\car LG\times^{L^+P_i}\bP(V_i)$ is an isomorphism.
\end{lemma}
\Pf It suffices to check that the morphism is bijective on fibres \cite[IV.4.2.E(3)]{KumarFlagbook}. By $\LG$-equivariance we may check this over the base point. However, the stabilizer of $v_i$ in $L^+P_i$ equals $L^+B$ and since $V_i$ is irreducible, this implies $L^+P_i/L^+B \simeq \bP(V_i)$.
\qed

\vskip8pt

In \cite[p. 54]{Faltings} Faltings constructs a central extension
$$ 1\longrightarrow \bG_m\longrightarrow \bar{L}G\stackrel{\psi}{\longrightarrow} LG\longrightarrow 1$$
which acts on all line bundles on $\cF$ (denoted $\tilde{L}G$ in loc.cit.). The copy of $\bG_m$ corresponds to the central element $c$ in the Kac-Moody setting. We consider $T$ as a subgroup of $LG$ and consider its inverse image $\bar{T}$ in $\bar{L}G$, i.e. $\bar{T}:=\psi^{-1}(T).$ The central extension has a unique splitting over $L^+B$ and we regard $T$ as being contained in $\bar{T}$. There is a fundamental character $\rho_i\in X^*(\bar{T})$ whose restriction to $T$ gives the character of $V_i$.

\vskip8pt
On the other hand, there is a natural identification $$LG\times^{L^+P_i}\bP(V_i)\simeq \bP(LG\times^{L^+P_i}V_i)$$ and we have the family
of line bundles $\cO(n)$ for $n\in\bbZ$ on this space. Each character $\lambda\in X^*(\bar{T})$ yields an associated line bundle $\cL_\lambda=LG\times^{L^+B} k_\lambda$ on $\cF$ \cite[10.a.]{PR_twisted}. Here, $k_\lambda$ denotes the one-dimensional $\bar{T}$-representation induced by $\lambda$.

 \begin{lemma}\label{assertionII}
Let $0\leq i \leq l$. The pull-back of $\cO(-1)$ to $\cF$ equals $\cL_{\rho_i}$.
The $\cL_{\rho_i}$ form a basis of the Picard group $Pic(\cF)$ of $\cF$.
\end{lemma}
\begin{proof}

Consider the two-dimensional $L^+P_i$-representation $V_i$ of weight $\rho_i$ and highest weight vector $v_i\in V_i$ as above. Consider the line bundle $\cL_{\rho_i}=G\times^{L^+B} kv_i$ on $\cF$ and fix a point $gL^+B\in\cF$. The fibre of $\cL_{\rho_i}$ in $gL^+B$ equals $[g,v_i]$. Since  the isomorphism $\cF\simeq\bP(G\times^{L^+P_i}V_i)$ takes $gL^+B$ to $[g,v_i]$, we see that $\cL_{\rho_i}$ equals the pull-back of the tautological line bundle $\cO(-1)$. 
The second assertion follows from this together with \cite[Prop. 10.1]{PR_twisted}.
\end{proof}

\vskip8pt

We need to introduce yet another torus. Let $P$ be either $B$ or $P_i$ or $G$ and abbreviate $\cF_P=LG/L^+P$. For a given $k$-algebra $R$, its group of units $a\in R^\times$ acts on the rings $R[[t]]$ and $R((t))$ via $t^m\mapsto a^mt^m$ ('turning the loop'). This induces an action of $\bG_m$ on $LG$ preserving the subgroup $L^+P$. We obtain an action on $\cF_P$ which commutes with the action of $T$. We write $$\hT:=\bG_m\times T$$
for the extended torus. This copy of $\bG_m$ corresponds to the derivation element $d$ in the Kac-Moody setting. The action of the extended torus preserves the Schubert cell $C_w$. Indeed, we may write $w=w'e^\lambda$ with $w'\in W_0$ and $\lambda\in\Lambda$. Given $a\in R^\times$ one has $$a(w*)=a(w')a(\lambda(t))*=w'\lambda(at)*=w'\lambda(a)\lambda(t)*=\lambda(a)w*=w*$$
since $\lambda(a)\in T(R)$. Therefore $\hT$ indeed stabilizes $C_w$. It induces an action of $\hT$ on the Schubert variety $S_w$. Even more, $w*$ is the only fixpoint in $C_w$ for the action of $T$ and $\hT$. Namely, the cell $C_w$ is isomorphic to a finite direct sum over copies of 'root spaces' $\bA^1_k$ for $\hT$ indexed by affine roots $(\alpha,m)$ and with 'root vectors' $t^mX_{\alpha}$. The isomorphism is $\hT$-equivariant and maps $w*$ to the origin in $\oplus \bA^{1}_k$. Here, $X_\alpha$ denotes a root vector for the root group $U_{\alpha}\subset G$ and the torus $\hT$ acts on $t^mX_{\alpha}$ through the character $(a,s)\mapsto a^m\alpha(s)$, e.g. \cite[p. 46/53]{Faltings}.

\section{Equivariant $K$-theory}

\subsection{Schubert classes}

Let $P$ be either $B$ or $P_i$ or $G$ and let $\cF_P=LG/L^+P$.
Suppose for a moment that $S\subset \cF_P$ is a $\hT$-stable ind-subvariety which is, in fact, a finite-dimensional algebraic $k$-variety.
We then have the abelian category of $\hT$-equivariant coherent modules on $S$. We denote the corresponding Grothendieck group by $K'_{\hT}(S)$.
Let $f:Y\rightarrow S$ be an equivariant morphism to $S$ from another $\hT$-variety $Y$. If $f$ is flat, then pull-back of modules induces a homomorphism $f^*: K'_{\hT}(S)\rightarrow K'_{\hT}(Y)$. If $f$ is proper, there is a push-forward homomorphism $f_*: K'_{\hT}(Y)\rightarrow K'_{\hT}(S)$ induced by $\sum_j (-1)^jR^jf_*$ \cite[1.10/11]{Thomason_LefI}. Note that if $f$ is actually a closed embedding, then of course $R^jf_*=0$ for $j>0$. Applying this to the structure map of $S$, the tensor product endowes the group $K'_{\hT}(S)$ with a module structure for the group ring
$$\bbZ[\hT]:=\bbZ[X^*(\hT)]=K'_{\hT}(\Spec\; k).$$
Let $H\subseteq\hT$ be a closed subgroup. Restricting the $\hT$-action to $H$ induces a linear homomorphism
\begin{numequation}\label{merkur} \bbZ[H]\otimes_{\bbZ[\hT]} K'_{\hT}(\cF)\longrightarrow K'_{H}(\cF).\end{numequation}
It is often bijective \cite{MerkurjevNotes}.
\vskip8pt

We define the Grothendieck group of $\cF_P$ to be
$$K'_{\hT}(\cF_P):=
\varinjlim_n K'_{\hT}(\cF^n_P)= \varinjlim_{w\in W^P}  K'_{\hT}(S_w)$$
  in analogy with the case of the complex Kashiwara affine flag manifold ($k=\bbC$) treated in \cite[2.2.1]{VV10}. The transition maps here are induced by the push-forward along the $\hT$-equivariant closed embeddings $\cF^n_P\hookrightarrow\cF^{n+1}_P$ resp. $S_u\hookrightarrow S_{w}$ for $u\leq w$.

  \vskip8pt

  Let $U_P^n:=\cF_P^n\setminus\cF_P^{n-1}$ with $U_P^0=\Spec\; k$. We let $j_n: U_P^n\hookrightarrow \cF^n_P$ denote the corresponding open embedding.
  \begin{lemma} There is a linear isomorphism $K'_{\hT}(U_P^n)\simeq\oplus_{\ell(w)=n} \bbZ[\hT]$.
  \end{lemma}
  \Pf We have $U_P^n=\coprod_{\ell(w)=n} C_w$ with a {\it linear} $\hT$-action on $C_w\simeq\bA_k^{\ell(w)}$, cf. discussion at the end of previous section. So the claim follows from homotopy invariance of $K'_{\hT}$ \cite[4.2]{Thomason_gpscheme}. \qed
  \vskip8pt
  We have the element $[\cO_{S_w}]$ in $K'_{\hT}(\cF_P)$ for $w\in W^P$. In the complex Kac-Moody setting these Schubert classes give rise to
  '$\bbZ$-bases' in the corresponding $K$-theories of the affine flag manifold \cite{KashiwaraFlag}, \cite{KoKu2}. In the following, we are rather interested in certain fields of positive characteristic $p>0$ where we are not aware of corresponding facts in the literature. The following results have straightforward proofs and are sufficient for our purposes.
  We let $$\bbK'_H(\cdot):=K'_{H}(\cdot)\otimes_{\bbZ} k\hskip10pt {\rm and }\hskip10pt k[\cdot]:=\bbZ[\cdot]\otimes_{\bbZ} k$$ 
  for any closed subgroup $H\subseteq\hT$. If $H=1$, we write $\bbK'$ instead of $\bbK'_{\{1\}}$. We denote by $\overline{\bbF}_p$ the algebraic closure
  of the finite field $\bbF_p$ with $p$ elements.
  \begin{prop}\label{prop-verify-merkur} Let $k\subseteq\overline{\bbF}_p$. The $k[\hT]$-module $\KTs(\cF^n_P)$ is free on the basis given by the Schubert classes $[\cO_{S_w}]$ with $\ell(w)=n$. If $H\subseteq \hT$ is closed subgroup, then {\rm (\ref{merkur})} is an isomorphism after tensoring with $k$.
  \end{prop}
 \Pf We have $\bbZ[H]\otimes_\bbZ K^{'\bullet}=K_{H}^{'\bullet}$ for those $H$-varieties which have trivial $H$-action and where $K^{'\bullet}$ denotes ordinary $K'$-theory.
  Applying this to $\bullet=1$ and using
  $$\bbK^1(\Spec\; \ke)={\ke}^\times\otimes_{\bbZ} {\ke}=0,$$ the first equivariant $K'$-group of $U_P^n=\coprod_{\ell(w)=n} C_w$ thus vanishes upon tensoring with $\ke$.  The localization sequence, tensored with $\ke$,
  $$  0 \longrightarrow\bbK'_H(\cF^{n-1}_P)\longrightarrow \bbK'_H(\cF^n_P)\stackrel{j_n^*}{\longrightarrow}\bbK'_H(U^n_P)\longrightarrow 0$$
  is therefore exact \cite[2.7]{Thomason_gpscheme}. Since a preimage under $j_n^*$ of the unit vector at the $w$-th position in $\oplus_{\ell(w)=n} \bbZ[H]$ is given by $[\cO_{S_w}]$, the first assertion follows now by induction on $n$ with $H=\hT$. The bijectivity of (\ref{merkur}) follows now easily from this. \qed

  \vskip8pt
  \begin{cor}\label{cor-basis} Let $k\subseteq\overline{\bbF}_p$. The $k[\hT]$-module $\KTs(\cF_P)$ is free on the basis given by the $[\cO_{S_w}]$ with $w\in W^P$.
  \end{cor}

\vskip8pt
Remark: The ind-structure on $\cF_P$ may also be defined via the $(P,P)$-Schubert varieties equal to the closure of the $L^+P$-orbits in $\cF_P$ \cite{PR_twisted}. Since an $L^+P$-orbit is not a topological cell anymore, we chose $(B,P)$-Schubert varieties to calculate the equivariant $K'$-theory of $\cF_P$.
  \vskip8pt
 For each $n$ the set of fixpoints $(\cF^n_P)^{\hT}$ is a non-empty closed subvariety of $\cF_P^n$. We let $\iota_n: (\cF^n_P)^{\hT}\hookrightarrow\cF_n$ be the corresponding closed embedding. We have the following simple version of the Thomason concentration theorem for fixed points varieties \cite[Thm. 2.1]{ThomasonDuke}.

 \begin{prop}\label{cor-iso}
 Let $k\subseteq\overline{\bbF}_p$.
The variety $(\cF_P^n)^{\hT}$ equals the union of the finitely many points $w*\in\cF_P$ for $\ell(w)\leq n$ and $w\in W^P$.
One has a linear isomorphism $$\iota_{n*}:\KTs((\cF_P^n)^{\hT}) \car\KTs(\cF_P^n)$$ which maps the class of the point $w*$ to $[\cO_{S_w}]$.

\end{prop}
\Pf One has $$(\cF_P^n)^{\hT}=\coprod_{\ell(w)\leq n} (C_w)^{\hT}=\coprod_{\ell(w)\leq n}\{w*\}$$ by the discussion at the end of the previous section. For the second assertion we consider the cartesian diagramm

$$\xymatrix{
(U_P^n)^{\hT}\ar[r]^{j_n} \ar[d]^{\tilde{\iota}_n} & (\cF_P^n)^{\hT} \ar[d]^{\iota_n}\\
  U_P^n\ar[r]^{j_n}& \cF_P^n.
}$$
Flat equivariant base change for equivariant $K'$-theory
yields $j_{n}^*\iota_{n*}=\tilde{\iota}_{n*}j_n^*$
 \cite[1.11]{Thomason_LefI}. We obtain a commutative diagram

 $$\xymatrix{
0\ar[r]  & \KTs((\cF_P^{n-1})^{\hT}) \ar[r] \ar[d]^{\iota_{n-1*}} & \KTs((\cF_P^n)^{\hT}) \ar[r] \ar[d]^{\iota_{n*}} & \KTs((U_P^n)^{\hT}) \ar[r] \ar[d]_{\simeq}^{\tilde{\iota}_{n*}} & 0\\
 0\ar[r] & \KTs(\cF_P^{n-1}) \ar[r] &  \KTs(\cF_P^n)\ar[r] & \KTs(U_P^n)\ar[r] & 0.
}$$
where the left-hand square commutes by functoriality for proper maps \cite[1.11]{Thomason_LefI} and the right-hand isomorphism comes from homotopy invariance. The lower horizontal sequence is exact by the proof of the preceding proposition and this implies the exactness of the upper horizontal sequence. The result follows now by induction on $n$.
\qed

  \subsection{Demazure operators}
We fix a number $n$ and consider the preimage of $\cF^n_i$ under the projection $\pi_i: \cF\rightarrow\cF_i$ from \ref{projection}. Since $\pi_i$ is flat and proper we have the homomorphisms
$$\xymatrix{
 K'_{\hT}(\pi_i^{-1}\cF^n_i) \ar@<1ex>[r]^>>>{\pi_{i*}}
&  K'_{\hT}(\cF^n_i) \ar@<1ex>[l]^<<<{\pi_i^*}
}.$$
In analogy to the case of the finite flag variety \cite{DemazureSchubert} and the complex Kashiwara affine flag manifold \cite{KashiwaraFlag}, we define a linear operator on $K'_{\hT}(\pi_i^{-1}\cF^n_i)$ by
$$ D_i:=\pi_i^*\circ\pi_{i*}.$$
\begin{lemma}\label{lem-comp}
The operator $D_i$ extends to an operator on $K'_{\hT}(\cF).$
\end{lemma}
\Pf Since $\pi_i$ is a principal $\bP^1$-bundle, each $\pi_i^{-1}\cF_i^{n}$ is a closed algebraic subvariety in $\cF$ and these varieties define, for varying $n$, the ind-structure on $\cF$. It suffices therefore to show that $D_i$ is compatible with the transition map $K'_{\hT}(\pi_i^{-1}\cF_i^{n-1})\rightarrow K'_{\hT}(\pi_i^{-1}\cF_i^{n})$.
We consider the cartesian diagram

$$\xymatrix{
\pi_i^{-1}\cF_i^{n-1}\ar[r]^{\pi_i} \ar[d]^{\tilde{\iota}} & \cF_i^{n-1}\ar[d]^{\iota}\\
  \pi_i^{-1}\cF_i^n \ar[r]^{\pi_i}& \cF_i^n
}$$
where the map $\iota$ is the natural inclusion and the map $\tilde{\iota}$ is induced by it.
Flat equivariant base change for equivariant $K'$-theory
yields $\pi_{i}^*\iota_*=\tilde{\iota}_*\pi_i^*$
and functoriality for proper maps yields $\iota_*\pi_{i_*}=\pi_{i*}\tilde{\iota}_*$. Hence
$$ \tilde{\iota}_*\pi^*_i\pi_{i*}=\pi_{i}^*\iota_*\tilde{\pi}_{i*}=\pi^*_i\pi_{i*}\tilde{\iota}_*.$$ \qed

\vskip8pt

We denote the resulting operator on $K'_{\hT}(\cF)$ by the same symbol $D_i$.
As in \cite[Prop. 2.6]{DemazureSchubert}, the projective bundle theorem in $K$-theory identifies the operator $D_i$ as an idempotent projector as follows. We note that under the isomorphism $\cF\simeq \bP(LG\times^{L^+P_i} V_i)$ of \ref{lem-ident} the morphism $\pi_i:\cF\rightarrow\cF_i$ becomes the projective space bundle associated to the rank $2$ vector bundle $LG\times^{L^+P_i} V_i\rightarrow \cF_i, [g,v]\mapsto gL^+P_i$.
\begin{lemma}
Let $0\leq i\leq l$. The correspondence
$$ (a_0,a_1)\mapsto 1\otimes \pi_i^*a_0 + \cL_{\rho_i}\otimes \pi_i^*a_1$$
induces a $\bbZ[\hT]$-linear isomorphism
$$K'_{\hT}(\cF^n_i)^{\oplus 2}\car K'_{\hT}(\pi_i^{-1}\cF^n_i)$$
for any $n$.
\end{lemma}
\Pf Since $\cL_{\rho_i}=\cO(-1)$ by \ref{assertionII}, we apply the (equivariant) projective bundle theorem \cite[Thm. 3.1]{Thomason_gpscheme} to $\pi_i^{-1}\cF^n_i\rightarrow \cF_i^n$. \qed

\vskip5pt

\begin{prop} In the situation of the lemma, the operator $D_i$ on $K'_{\hT}(\pi_i^{-1}\cF^n_i)$ has the property
$$D_i (1\otimes \pi_i^*a_0 + \cL_{\rho_i}\otimes \pi_i^*a_1) = 1\otimes \pi_i^*a_0.$$
\end{prop}
\Pf According to \cite[Lemma 8.1.1(c)]{QuillenH} the modules $M_0:=\pi_i^*a_0$ and $M_1(-1):=\cO(-1)\otimes M_1$ with $M_1:=\pi_i^*a_1$  satisfy $\pi_{i*}M_0=a_0$ and $\pi_{i*}M_1(-1)=0$ and, moreover, they are regular modules in the sense of Mumford, cf. loc.cit. Then \cite[Lemma 8.1.3]{QuillenH} implies $R^j\pi_{i*}M_0=R^j\pi_{i*}M_1(-1)=0$ for $j>0$.\qed

\vskip8pt
\begin{cor}\label{cor-idem} The operator $D_i$ on $K'_{\hT}(\cF)$ satisfies $D_i^2=D_i$.
\end{cor}

\vskip8pt
Remark: According to the preceding proposition, $D_i$ may be also be viewed as an algebraic characteristic $p$ analogue of the operator $D_i$ defined by Kostant-Kumar
on the topological equivariant $K$-theory of the complex analytic Kac-Moody flag manifold \cite[Def. (3.6)]{KoKu2}.
\vskip8pt

We compute the action of $D_i$ on the elements $[\cO_{S_w}]$ of $K'_{\hT}(\cF)$.
\begin{lemma}\label{lem-aux} Let $Z\subset \cF_i^n$ be a closed $\hT$-subvariety and let $[\cO_Z]$ be its class in $K'_{\hT}(\cF_i^n)$. Then
$\pi^*[\cO_Z]=[\cO_{\pi_i^{-1}Z}]$ in $K'_{\hT}(\pi_i^{-1}\cF_i^n)$.
\end{lemma}
\Pf This follows directly from the flatness of $\pi_i$. \qed

\vskip8pt
For $w\in W^{P_i}$ let $S_w^{i}$ be a Schubert variety in $\cF_i$. We have $\pi_i^{-1}S_w^{i}=S_{w'}$
where $w'\in W$ has the property: $w'=w$ if $ws_i<w$ or else $w'=ws_i$. This implies that for any $w\in W$ one has
\begin{numequation}\label{inverse_schubert} \pi_i^{-1}\pi_i(S_{w})=S_{w} \hskip10pt{\rm if}\hskip10pt ws_i<w \hskip10pt {\rm or~else}\hskip10pt \pi_i^{-1}\pi_i(S_{w})=S_{ws_i}.\end{numequation}

Remark: Similarly, the inverse image of $S_w\subset\cG$ under $\pi$ equals $S_{ww_0}\subset \cF$ where $w_0$ denotes the longest element in $W_0$.
Indeed, the inverse image $\pi^{-1}S_w$ is a Schubert variety in $\cF$ and $ww_0$ is the representative of maximal length in its coset modulo $W_0$.

\begin{lemma} Let $w\in W$ and suppose $S_w\subset \pi_i^{-1}\cF_i^n$ for some $n$. Then $$\pi_{i*}[\cO_{S_w}]=[\cO_{\pi_i(S_w)}]$$
in $K'_{\hT}(\cF^n_i)$.
\end{lemma}
\Pf Consider the commutative diagram
$$\xymatrix{
S_w\ar[r]^{\tilde{\iota}} \ar[d]^{\pi_i} & \pi_i^{-1}\cF_i^{n}\ar[d]^{\pi_i}\\
  \pi_i(S_w) \ar[r]^{\iota}& \cF_i^n.
}$$
where $\iota$ and $\tilde{\iota}$ are the corresponding closed embeddings. Functoriality for proper maps implies $\pi_{i*}\tilde{\iota}_*=\iota_*\pi_{i*}$ on $K'_{\hT}(S_w)$. However, $\iota_*$ and $\tilde{\iota}_{*}$ are exact functors on coherent modules, so that we are reduced to show: $$R^j\pi_{i*}(\cO_{S_w})=0 \hskip10pt{\rm for}\hskip10pt j>0 \hskip10pt{\rm and}\hskip10pt \pi_{i*}(\cO_{S_w})=\cO_{\pi_i(S_w)}.$$
We distinguish two cases. Assume first that $ws_i<w$. Then $\pi_i^{-1}\pi_i(S_{w})=S_{w}$ by \ref{inverse_schubert} so that our commutative diagramm is even cartesian. Since $\pi_i$ is a projective space bundle, our claims follow from \cite[8.1.1/8.1.3]{QuillenH}. Assume now $ws_i>w$. Then
$\pi_i^{-1}\pi_i(S_{w})=S_{ws_i}$ by \ref{inverse_schubert} and the map $\pi_i: S_{ws_i}\rightarrow \pi_i(S_{w})$ is a proper birational morphism \cite[Remark 3.2.3]{BrionKumar}. Our claims follow since $S_{ws_i}$ and $\pi_i(S_w)$ are normal with at most rational singularities \cite[Prop. 9.6]{PR_twisted}.
\qed

\vskip8pt

\begin{prop}\label{prop-explicit} Let $w\in W$. One has
$$D_i[\cO_{S_w}] =\begin{cases} [\cO_{S_w}]& {\rm if~} ws_i<w \\ [\cO_{S_{ws_i}}] & {\rm if~}ws_i>w. \end{cases}$$
\end{prop}
\Pf According to the two preceding lemmas we have $\pi_i^*\pi_{i*}[\cO_{S_w}]=[\cO_{\pi_i^{-1}\pi_i(S_w)}]$ so that \ref{inverse_schubert} implies the assertion.
\qed

\vskip8pt
From now on we assume that $k\subseteq \overline{\bbF}_p$. We consider the endomorphism ring $$\End_{k[\hT]}(\KTs(\cF))^{\rm op}$$
of $k[\hT]$-linear operators on $\KTs(\cF)$ with multiplication $(\phi\varphi)(c)=\varphi(\phi(c))$ for $c\in\KTs(\cF)$. In particular, $\KTs(\cF)$
is naturally a {\it right} module over this ring. For an element $w\in W$ with reduced decomposition $w=s_{i_1}\cdot\cdot\cdot s_{i_n}$ we define a linear operator $D_w$ on $\KTs(\cF)$ by
$$D_w:=D_{i_1}\cdot\cdot\cdot D_{i_n}$$
and $D_1:={\rm id}$.
\begin{lemma}\label{lem-welldefined}
The operator $D_w$ does not depend on the choice of reduced decomposition for $w$.
\end{lemma}
\Pf According to \cite[IV.\S1.Prop. 5]{B-L2}, it suffices to check for $i\neq j$ and $m_{ij}=ord(s_is_j)<\infty$ the relations $D_iD_jD_i\cdot\cdot\cdot=D_jD_iD_j\cdot\cdot\cdot$ with $m_{ij}$ factors on both sides. According to \ref{cor-basis} we can check this on Schubert classes. Since $W$ is crystallographic, we have $m_{ij}\in\{2,3,4,6\}$, so that the claim can be checked case-by-case. We treat the case $m_{ij}=2$, the other cases follow along the same lines, but with more notation. 
Since $m_{ij}=2$, we must show that $D_i$ and $D_j$ commute. So let $w\in W$ and let us argue case-by-case. First, if $ws_i, ws_j<w$, then $[\cO_{S_w}]$ is a fixpoint of $D_iD_j$ and $D_jD_i$.
Secondly, if $ws_i, ws_j>w$, then $ws_is_j>ws_j$ and $ws_js_i>ws_i$ according to case (b) in the proof of \cite[Prop. 5.9]{HumphreysCoxeter} which yields
$$ [\cO_{S_w}](D_iD_j)=[\cO_{S_{ws_i}}]D_j=[\cO_{S_{ws_is_j}}]=[\cO_{S_{ws_js_i}}]=[\cO_{S_w}](D_jD_i).$$
By symmetry in $i$ and $j$, it now remains to consider the case $ws_j>w$ and $ws_i<w$. Then $[\cO_{S_w}](D_iD_j)=[\cO_{S_{ws_j}}]$ and $[\cO_{S_w}](D_jD_i)=[\cO_{S_{ws_j}}]D_i$. We have either $ws_is_j<ws_j$ or $ws_is_j\leq w$ according to
\cite[Prop. 5.9]{HumphreysCoxeter}. In the first case, we are done. The second case leads via $ws_is_j\leq w<ws_j$ to the first case. \qed

\vskip8pt
We summarize the properties of the $D_w$. Let $\cS$ be the set of generators $s_i$ for the affine Coxeter group $W$.
\begin{cor}\label{cor-demrel}
Let $w,w'\in W$. The linear operators $D_w$ on $\KTs(\cF)$ satisfy the relations
\vskip2pt
$$\begin{array}{ccll}
D_wD_{w'} & = & D_{ww'} & {\rm if~}\ell(w)+\ell(w')=\ell(ww'),\\
 &  &  &\\
D_{s}^2 & = & D_{s} & {\rm for~ any~}s\in \cS.
\end{array}
$$
\end{cor}
\vskip8pt
Remark: Since $D_w$ acts linearly on $\KTs(\cF)$, there is an induced action of $D_w$ on
$$\bbK_H'(\cF)\simeq k[H] \otimes_{k[\hT]} \KTs(\cF)$$
for any closed subgroup $H\subseteq\hT$ according to \ref{prop-verify-merkur}.

\vskip8pt

Remark: The projection $\pi: \cF\rightarrow\cG$ is a locally trivial $G/{\rm pr}(\cB)$-fibration and we have the linear homomorphism
$$\pi^*: K'_{\hT}(S_w)\rightarrow  K'_{\hT}(\pi^{-1}S_w).$$
Flat equivariant base change, as in the proof of \ref{lem-comp}, shows that these homomorphisms commute with transition maps and give in the limit a linear homomorphism $\pi^*: K'_{\hT}(\cG)\rightarrow  K'_{\hT}(\cF).$ It follows from \ref{lem-aux} and the subsequent remark that $\pi^*$ maps the class of $S_w\subset\cG$ to the class of $S_{ww_0}\subset \cF$.
  According to \ref{cor-basis}, the induced morphism 
\begin{numequation}\label{equ-satakeembedding}\pi^*: \KTs(\cG)\hookrightarrow  \KTs(\cF)\end{numequation}
is therefore injective.

\subsection{Iwahori-Hecke algebras}
We keep the assumption $k\subseteq\overline{\bbF}_p$.
The integral {\it Coxeter-Hecke algebra} of $(W,S)$ with weights $q_s=0$ for all $s\in \cS$ is given as $$\sH=\oplus_{w\in W}\;\bbZ\tilde{Y}_w$$ with relations

$$\begin{array}{ccll}
\tilde{Y}_w\tilde{Y}_{w'} & = & \tilde{Y}_{ww'} & {\rm if~}\ell(w)+\ell(w')=\ell(ww'),\\
 &  &  &\\
\tilde{Y}_{s}^2 & = & -\tilde{Y}_{s} & {\rm for~ any~}s\in \cS,
\end{array}
$$

cf. \cite[Ex. IV.\S2.(24)]{B-L2}. We let $Y_w:=(-1)^{\ell(w)}\tilde{Y}_w$ for any $w$ and $\sH_R:=\sH\otimes_{\bbZ} R$ for a commutative ring $R$.
We have the monoid $$\Lambda_+=\{\lambda\in\Lambda: \langle \lambda,\alpha\rangle \geq 0 {\rm ~for~all~}~\alpha\in\Phi^+\}$$
of dominant cocharacters in $\Lambda$ and its corresponding monoid ring $\bbZ[\Lambda_+]$.
One has $\ell(e^\lambda)=\langle \lambda,2\rho\rangle$ for any $\lambda\in\Lambda_+$ where $2\rho$ equals the sum over elements from $\Phi^+$.
Hence
$$ \Theta: \bbZ[\Lambda_+]\longrightarrow\sH, \lambda\mapsto Y_{e^\lambda}$$
is an injective ring homomorphism. In this way, we view $\bbZ[\Lambda_+]$ as a subring of $\sH$.

\vskip8pt

\begin{prop}
Letting $Y_w$ act through $D_w$ makes $\bbK_{\hT}'(\cF)$ a right $\sH_{k[\hT]}$-module.
\end{prop}
\Pf The elements $Y_w$ are a $\bbZ$-basis of $\sH$ with defining relations $Y_wY_{w'}=Y_{ww'}$ if $\ell(w)+\ell(w')=\ell(ww')$ and $Y_{s}^2=Y_{s}$
for any $s\in \cS$. Acccording to \ref{cor-demrel} the
map $Y_w\mapsto D_w$ extends therefore to a ring homomorphism between $\sH_{k[\hT]}$ and $\End_{k[\hT]}(\KTs(\cF))^{\rm op}$. \qed

\vskip8pt
We view the algebra $\sH_{k[\hT]}$ via its multiplication as a right module over itself. 
\begin{thm}
There is an isomorphism of right $\sH_{k[\hT]}$-modules
$$\Xi: \sH_{k[\hT]}\car \bbK_{\hT}'(\cF),~~Y_w\mapsto [\cO_{S_w}].$$
Specialising at $\hT=1$ yields an isomorphism of right $\sH_k$-modules $\sH_k\simeq\bbK'(\cF)$.

\end{thm}
\Pf The first map is $k[\hT]$-linear and bijective \ref{cor-basis}. It is equivariant for the action of any $Y_s$ by \ref{prop-explicit} and therefore for the action of any $Y_w$. The second statement follows from the remark after \ref{cor-demrel}.\qed

\vskip8pt

We show that the Grassmannian $\cG$ gives rise to a Hecke submodule in $\bbK_{\hT}'(\cF)$ for the image of $\Theta$.
One has $\ell( e^{-\lambda} w)=\ell(w^{-1}e^\lambda)=\ell(w)+\ell(e^{-\lambda})$ for any $\lambda\in\Lambda_+$ and $w\in W_0$. This means that
the elements from $\Lambda_-:=-\Lambda_+=w_0(\Lambda_+)$ are the minimal representatives in their corresponding cosets modulo $W_0$. We let $\pi^{*}_{-}$ be the
restriction of the homomorphism $\pi^*$ from (\ref{equ-satakeembedding}) to the submodule generated by the Schubert classes indexed by $\Lambda_-$ and consider its image ${\rm Im}\;\pi^{*}_-\subset \bbK'_{\hT}(\cF).$ Note that the ring $k[\Lambda_+]$ acts on $\bbK'_{\hT}(\cF)$ by restricting the $\sH_{k}$-action via $\Theta$.

\begin{cor} The isomorphism $\Xi$ induces an isomorphism $$ Y_{w_0} k[\Lambda_+]\car {\rm Im}\;\pi^{*}_{-}$$ as right $k[\Lambda_+]$-modules.
In particular, ${\rm Im}\;\pi^{*}_{-}$ is a free $k[\Lambda_+]$-module of rank $1$.
\end{cor}
\Pf If $\lambda\in\Lambda_-$, then $e^\lambda w_0=w_0e^{w_0(\lambda)}$ with $w_0(\lambda)\in\Lambda_+$. According to \ref{equ-satakeembedding} and the subsequent remarks, the $k[\hT]$-submodule ${\rm Im}\;\pi^{*}_{-}$ is therefore freely generated by the classes of
the $S_{w_0e^\lambda}$ where $\lambda\in\Lambda_+$. Given another $\mu\in\Lambda_+$, one has $\ell(w_0e^\lambda e^\mu)=\ell(w_0e^\lambda)+\ell(e^\mu)$
and therefore $D_{e^\mu}[\cO_{S_{w_0e^\lambda}}]= [\cO_{S_{w_0e^{\lambda+\mu}}}]\in {\rm Im}\;\pi^{*}_{-}$. This implies all assertions.
\qed

\vskip8pt

Remark: Suppose the group $G$ comes from a simply-connected semi-simple and simple group scheme $G$ over $\bbZ$. Let $F$ be a non-archimedean local field, i.e. equal to a finite extension of the field $\bbQ_p$ of $p$-adic numbers with residue field $k$ or else equal to $k((t))$, with $k\subseteq\overline{\bbF}_p$ finite. Let $o_F\subset F$ be its ring of integers. Let $I\subset G(F)$ be an Iwahori-subgroup of the group $G(F)$. The {\it Iwahori-Hecke algebra} of $G(F)$ over $k$ equals the convolution algebra $k[I\setminus G(F)/I]$ of $k$-valued functions with finite support on the double cosets of $G(F)$ mod $I$. A $k$-basis is given by the characteristic functions of the double cosets $I\setminus G(F)/I$ which, in turn, are in bijection with the group $W$. Denote by $\tau_w$ the characteristic function of $IwI, w\in W$.
By the Iwahori-Matsumoto presentation \cite{IwahoriMatsumoto} there is an algebra isomorphism $$k[I\setminus G(F)/I]\car \sH_k,~~\tau_w\mapsto \tilde{Y}_w.$$ We thus obtain the theorem mentioned in the introduction. Let $G_0=G(o_F)$ and consider the {\it spherical Hecke algebra} $k[G_0\setminus G(F) / G_0]$. In the $p$-adic case, the mod $p$ Satake map \cite{HerzigSatake} induces an algebra isomorphism
$$k[G_0\setminus G(F) / G_0]\car k[\Lambda_+].$$
For the compatibilies between the various isomorphisms we refer to \cite[8C.]{OllivierParabolic}.

\bibliographystyle{plain}
\bibliography{mybib}

\end{document}